\documentclass[11pt,a4paper]{amsart}

\usepackage{amsthm} 
\usepackage{amsmath} 
\usepackage[all]{xy} 
\usepackage{tensor} 
\usepackage{combelow} 
\usepackage{hyperref} 

\newtheorem{theorem}{Theorem}[section]
\newtheorem{lem}[theorem]{Lemma}
\newtheorem{prop}[theorem]{Proposition}
\newtheorem{defi}[theorem]{Definition}
\newtheorem{subsec}[theorem]{}

\renewcommand{\O}{\mathcal{O}}
\newcommand{\K}{\mathcal{K}}
\newcommand{\T}{\mathcal{T}}
\DeclareFontFamily{OT1}{pzc}{}
\DeclareFontShape{OT1}{pzc}{m}{it}{<->s*[1.10] pzcmi7t}{}
\DeclareMathAlphabet{\mathscr}{OT1}{pzc}{m}{it}
\renewcommand{\k}{\mathscr{k}}
\newcommand{\oG}{\bar{G}}
\newcommand{\End}{\operatorname{End}}
\newcommand{\og}{\bar{g}}
\newcommand{\tens}[1]{{{\otimes}_{#1}}}
\newcommand{\Ms}{M^{\ast}}

\title[Group graded endomorphism algebras]{Group graded endomorphism algebras\\and Morita equivalences}
\author[A. Marcus and V. A. Minu\cb{t}\u{a}]{Andrei Marcus {\rm and} Virgilius-Aurelian Minu\cb{t}\u{a}}

\begin{document}
\begin{abstract} We prove a  group graded Morita equivalences version of the ``butterfly theorem" on character triples. This gives a method to construct an equivalence between block extensions from another related equivalence.\\[0.1cm]
\end{abstract}

\maketitle

\section{Introduction and preliminaries}  \label{s:Notations}

The Butterfly theorem, as stated by B. Sp\"ath in \cite[Theorem 2.16]{ch:Spath2018}, gives the possibility to construct certain relations between character triples. The result is very useful in obtaining reduction methods for the local-global conjectures in modular representation theory of finite groups. In this paper we consider group graded Morita equivalences between block extensions, and we obtain  an analogue of \cite[Theorem 2.16]{ch:Spath2018}. Our main result Theorem \ref{th:butterfly} shows how to construct a group graded Morita equivalence from a given one, under very similar assumptions to those in \cite{ch:Spath2018}.

In general, our notations and assumptions are standard and follow \cite{book:Marcus1999}. To introduce our context, let $G$ be a finite group, $N$  a normal subgroup of  $G$, and denote by $\oG$ the factor group $G/N$. Let $A=\bigoplus_{\bar{g}\in\oG} A_{\bar{g}}$ be a strongly $\oG$-graded $\O$-algebra with the identity component $B:=A_1$, where $(\K,\O,\k)$ is a $p$-modular system. For a subgroup  $\bar{H}$ of $\oG$  we denote by
$A_{\bar{H}}:=\bigoplus_{\og\in\bar{H}}A_{\og}$ the truncation of $A$ from $\oG$ to $\bar{H}$.

For the sake of simplicity, in this article we will mostly consider only crossed products, also because the generalization of the statements to the case of strongly graded algebras is a mere technicality. Recall that if $A$ is a crossed product, we can chose  an invertible homogeneous element $u_{\og}$ in the component $A_{\og}$, for all $\og\in \oG$.

Our main example for a $\bar G$-graded crossed product is obtained as follows: Regard $\O G$ as a $\oG$-graded algebra with the 1-component $\O N$. Let $b\in Z(\O N)$  be a $\oG$-invariant block idempotent. We denote:
\[ A:=b\O G, \qquad  B:=b\O N. \]
Then the block extension $A$ is a $\oG$-graded crossed product, with  1-compo\-nent $B$.

The paper is organized as follows. In Section \S\ref{s:Group_graded_Morita_equivalences} we recall from \cite{book:Marcus1999} the main facts on group graded Morita equivalences, and we state a graded variant of the second Morita Theorem \cite[Theorem 12.12]{book:Faith1973}. In Section \S\ref{s:Centralizers} we show that there is a natural map, compatible with Morita equivalences, from the centralizer $C_A(B)$ of $B$ in $A$ to the endomorphism algebra of a $\bar G$-graded $A$-module induced from a $B$-module. In the last Section, Section \S\ref{s:Butterfly}, we prove that a Morita equivalence between the 1-components of two block extensions always lift to a graded equivalence between certain centralizer algebras. This is the main ingredient in the proof of our main result, Theorem \ref{th:butterfly}.

\section{Group graded Morita equivalences} \label{s:Group_graded_Morita_equivalences}

Let $A=\bigoplus_{\og\in\oG}A_{\og}$ and $A'=\bigoplus_{\og\in\oG}A'_{\og}$ be strongly $\oG$-graded algebras, with the 1-components $B$ and $B'$ respectively.

\begin{subsec}\normalfont  It is clear that $A\tens{\O} A'^{\mathrm{op}}$ is a $\oG\times \oG$-graded algebra. Let \[\delta(\oG):=\{(\og,\og)\mid \og\in \oG\}\] be the diagonal subgroup of $\bar G\times \bar G$, and let $\Delta$  be the diagonal subalgebra of $A\tens{\O} A'^{\mathrm{op}}$
\[\Delta:=(A\tens{\O} A'^{\mathrm{op}})_{\delta(\oG)}=\bigoplus_{\og\in\oG}A_{\bar{g}}\otimes A'_{\bar{g}^{-1}}.\]
Then $\Delta$ is a $\oG$-graded algebra, with  1-component $\Delta_1=B\tens{\O}B'^{\mathrm{op}}.$
\end{subsec}

\begin{subsec}\normalfont Let $M$ be a $(B,B')$-bimodule, or equivalently  $M$ is a $B\tens{\O}B'^{\mathrm{op}}$-module, thus a $\Delta_1$-module. Let  $M^{\ast}:=\mathrm{Hom}_B(M,B)$ be its $B$-dual. Note that if $B$ is a symmetric algebra, then we have the  isomorphism
\[M^{\ast}:=\mathrm{Hom}_B(M,B)\simeq \mathrm{Hom}_{\O}(M,\O),\] where $\mathrm{Hom}_{\O}(M,\O)$, is  the $\O$-dual of $M$.
\end{subsec}

\begin{defi} 	We say that the $\oG$-graded $(A,A')$-bimodule $\tilde{M}$ induces a $\oG$-graded Morita equivalence between $A$ and $A'$, if $\tilde{M}\tens{A'}\tilde{M}^{\ast}\cong A$ as $\oG$-graded $(A,A)$-bimodules and that $\tilde{M}^{\ast}\tens{A}\tilde{M}\cong A'$ as $\oG$-graded $(A',A')$-bimodules, where the $A$-dual $\tilde{M}^{\ast}=\text{Hom}_A(\tilde{M},A)$  of $\tilde{M}$ is a $\oG$-graded $(A',A)$-bimodule.
\end{defi}

By   \cite[Theorem 5.1.2]{book:Marcus1999}, the following statements are equivalent:
\begin{enumerate}
\item between $B$ and $B'$ we have a Morita equivalence given by the $\Delta_1$-module $M$, and $M$ extends to a $\Delta$-module.
\item $\tilde{M}:=A\tens{B}M$ is a $\oG$-graded $(A,A')$-bimodule, and $\tilde{M}^{\ast}:=A'\tens{B'}\Ms$ is a $\oG$-graded $(A',A)$-bimodule, which induce a $\oG$-graded Morita equivalence between $A$ and $A'$, given by the functors:
\[\xymatrix@C+=3cm{
		A \ar@<+.5ex>@{<-}[r]^{ \tensor*[_{A}]{\tilde{M}}{_{A'}}\tens{A'}-} & A'. \ar@<+.5ex>@{<-}[l]^{ \tensor*[_{A'}]{\tilde{M}}{_{A}^{\ast}}\tens{A}-}
	}\]
\end{enumerate}

In this case, by \cite[Lemma 1.6.3]{book:Marcus1999}, we have the natural isomorphisms of $\bar G$-graded bimodules
\[\tilde{M}:=A\tens{B}M\simeq M\tens{B'}A'\simeq ((A\tens{\O}A'^{\mathrm{op}})\tens{\Delta}M).\]

\begin{subsec}\normalfont Assume that $B$ and $B'$ are Morita equivalent. Then, by the second Morita Theorem \cite[Theorem 12.12]{book:Faith1973}, we can choose the bimodule isomorphisms
\[\varphi:M^{\ast}\tens{B}M\to B', \qquad \psi:M\tens{B'}M^{\ast}\to B.\]
such that
\[\psi(m\tens{} m^{\ast})n=m\varphi(m^{\ast}\tens{} n),\quad\forall m,n\in M,\ m^{\ast}\in M^{\ast}\]
and that
\[\varphi(m^{\ast}\tens{} m)n^{\ast}=m^{\ast}\psi(m\tens{} n^{\ast}),\quad\forall m^{\ast},n^{\ast}\in M^{\ast},\ m\in M.\]

By the surjectivity of this functions, we may choose finite sets $I$ and $J$ and the elements $m_j^{\ast},\,n_i^{\ast}\in M^{\ast}$ and $m_j,\,n_i\in M$, for all  $i\in I,\ j\in J$ such that:
\[\varphi(\sum\limits_{j\in J}{m_j^{\ast}\tens{B}m_j})=1_{B'},\qquad \psi(\sum\limits_{i\in I}{n_i\tens{B}n_i^{\ast}})=1_{B}.\]
\end{subsec}

\begin{subsec}\normalfont Assume that  $\tilde{M}$ and $\tilde{M}^{\ast}$ give a $\oG$-graded Morita equivalence between $A$ and $A'$. As above, by  \cite[Theorem 12.12]{book:Faith1973}, we can choose  the isomorphisms
\[\tilde{\varphi}:\tilde{M}^{\ast}\tens{A}\tilde{M}\to A',\qquad  \tilde{\psi}:\tilde{M}\tens{A'}\tilde{M}^{\ast}\to A\]
of $\oG$-graded bimodules  such that
\[\tilde{\psi}(\tilde{m}\otimes \tilde{m}^{\ast})\tilde{n}=\tilde{m}\tilde{\varphi}(\tilde{m}^{\ast}\otimes \tilde{n}),\quad\forall \tilde{m},\tilde{n}\in \tilde{M},\ \tilde{m}^{\ast}\in \tilde{M}^{\ast}\]
and that
\[\tilde{\varphi}(\tilde{m}^{\ast}\otimes \tilde{m})\tilde{n}^{\ast}=\tilde{m}^{\ast}\tilde{\psi}(\tilde{m}\otimes \tilde{n}^{\ast}),\quad\forall \tilde{m}^{\ast},\tilde{n}^{\ast}\in \tilde{M}^{\ast},\ \tilde{m}\in \tilde{M}.\]
Actually, $\tilde{\varphi}_1$ and $\tilde{\psi}_1$ are the same with $\varphi$ and $\psi$ from before, and are $\Delta$-linear isomorphisms. Moreover, we have that $1_A=1_B\in B$ and $1_{A'}=1_{B'}\in B'$. Henceforth, we may choose the same finite sets $I$ and $J$ and the same elements $m_j^{\ast},\,n_i^{\ast}\in M^{\ast}$ and $m_j,\,n_i\in M$, $\forall i\in I,j\in J$ such that:
\[\tilde\varphi(\sum_{j\in J}{m_j^{\ast}\tens{B}m_j})=1_{B'}, \qquad \tilde\psi(\sum_{i\in I}{n_i\tens{B}n_i^{\ast}})=1_{B}.\]
\end{subsec}


\section{Centralizers and graded endomorphism algebras} \label{s:Centralizers}

\begin{subsec}\normalfont We will assume that $A$ and $A'$ are $\bar G$-graded crossed products, although the results of this section can be generalized to strongly graded algebras.  Let $U\in B$-mod and $U'\in B'$-mod such that $U'=M^{\ast}\tens{B}U$. We denote
\[E(U):=\End(A\tens{B}U)^{\mathrm{op}},\qquad 	E(U'):=\End(A'\tens{B'}U')^{\mathrm{op}},\]
the $\oG$-graded endomorphism algebras of the modules induced from $U$ and $U'$.
\end{subsec}

We will prove  that there exists a natural $\oG$-graded algebra homomorphism between the centralizer of $B$ in $A$ and $E(U)$, compatible with $\bar G$-graded Morita equivalences.

\begin{lem} \label{prop:theta} The map
\[\theta:C_A(B)\to E(U),\qquad 	\theta(c)(a\tens{} u)=ac\tens{} u,\]
 where $c\in C_A(B)$, $a\in A$, and $u\in U$, is a  homomorphism of $\oG$-graded algebras.
\end{lem}

\begin{proof} We first need to show that the map is well-defined.  
For $c\in C_A(B)$, $a\in A$, $b\in B$ and $u\in U$, we have:
\[\theta(c)(ab\tens{B}u)=ab\cdot c\tens{B} u=a c b\tens{B} u = a c \tens{B} bu = \theta(c)(a\tens{B}bu).\]
To show that $\theta(c)$ is $A$-linear, let $a'\in A$; we have:
	\[\theta(c)(a'a\tens{B}u)=a'ac\tens{B}u=a'(ac\tens{B}u)=a'\theta(c)(a\tens{B}u).\]
To prove that the map is a ring homomorphism, let $c,c'\in C_A(B)$; we have:
\[\begin{array}{rcl}
	(\theta(c)\cdot\theta(c'))(a\tens{B}u)&=&(\theta(c')\circ\theta(c))(a\tens{B}u)\\
	&=&\theta(c')(\theta(c)(a\tens{B}u))\\
	&=&\theta(c')(ac\tens{B}u) = acc'\tens{B}u \\
    &=&\theta(cc')(a\tens{B}u).
\end{array}		\]
Finally, we check that $\theta$ is grade-preserving. Let $a_{\og}\tens{B}u\in A_{\og}\tens{B}U$ and $c\in C_A(B)_{\bar{h}}$, where $\og,\bar{h}\in\oG$. Then the definition of $\theta$ says that
\[\theta(c)(a_{\og}\tens{B}u)=a_{\og}\cdot c\tens{B} u\in A_{\og \bar h}\tens{B}U.\]
If follows that $\theta(c)$ belongs to $E(U)_{\bar h}$. The other properties are obvious.
\end{proof}

\begin{subsec}\normalfont By \cite[Lemma 1.6.3]{book:Marcus1999} we have
\[A\tens{B}M\simeq M\tens{B'}A',\]
and we will need an explicit isomorphism between the two. We will  choose  invertible elements $u_{\bar{g}} \in U(A)\cap A_{\bar{g}}$ and $u'_{\bar{g}}\in U(A)\cap A'_{\bar{g}}$ of degree ${\bar{g}}\in\oG$. We have that an arbitrary element $a'_{\bar{g}}\in A'_{\bar{g}}$ can be written uniquely in the  form $a'_{\bar{g}}=u'_{\bar{g}}b',$ where $b'\in B'.$ The desired $\oG$-graded bimodule isomorphism is:
\[\varepsilon:M\tens{B'}A'\to A\tens{B}M\qquad  m\tens{B'} a'_{\bar{g}}\mapsto    u_{\bar{g}}\tens{B}  u^{-1}_{\bar{g}} ma'_{\bar{g}}\]
for $m\in M$. We will also need the explicit isomorphism of $\bar G$-graded bimodules
\[\beta:A'\tens{B'} M^{\ast}\to M^{\ast}\tens{B}A\qquad  a'_{\bar{g}}\tens{B'} m^{\ast} \mapsto a'_{\bar{g}}m^*u_{\bar{g}}^{-1}\otimes_B u_{\bar{g}} \]
for $m^{\ast}\in M^{\ast}$. Henceforth we consider the  isomorphism of $\bar G$-graded $A'$-modules
\[\beta\tens{B}id_U:A'\tens{B'} M^{\ast}\tens{B}U\to M^{\ast}\tens{B}A\tens{B} U.\]
\end{subsec}

\begin{prop} 	Assume that $\tilde M$ and $\tilde{M}^{\ast}$ give a $\oG$-graded Morita equivalence between $A$ and $A'$. Then the diagram
\[\xymatrix@C+=3cm{
		E(U) \ar@{->}[r]^{\varphi_1}_{\sim} & E(U')\\
		C_A(B) \ar@{->}[r]^{\varphi_2}_{\sim} \ar@{->}[u]^{\theta} & C_{A'}(B') \ar@{->}[u]^{\theta'}  
}\]
is commutative, where the maps are defined as follows:
\begin{align*}
\theta(c)(a\tens{} u)&=ac\tens{} u, \\
 \theta'(c')(a'\tens{} u')&= a'c'\tens{} u' \\
 \varphi_1(f) &=(\beta\tens{B}id_U)^{-1}\circ(id_{\tilde{M}^{\ast}}\otimes f)\circ (\beta\tens{B}id_U), \\
 \varphi_2(c) &=\tilde{\varphi}(\sum\limits_{j\in J}{m_j^{\ast}c\tens{B}m_j}).
\end{align*}
for all  $a\in A,$ $a'\in A',$ $c\in C_A(B)$, $c'\in C_{A'}(B'),$  $u\in U$, $u'\in U'$ and $f\in E(U)$.
\end{prop}

\begin{proof} According to Lemma \ref{prop:theta}, we have that $\theta,\theta'$ are homomorphisms of $\oG$-graded algebras. Moreover,  $\varphi_1$ and $\varphi_2$ are the algebra isomorphisms induced by the $\bar G$-graded Morita equivalence.

To prove that the diagram is commutative, let $c\in C_A(B)_{\bar{h}}$, where $\bar{h}\in \bar G$. We consider arbitrary elements $a'_{\bar{g}}\in A'_{\bar{g}}$, where $\bar{g}\in\bar{G}$ and $u'=m^{\ast}\tens{B} u\in U'=M^{\ast}\tens{B}U$. By the above remarks, for all $f\in E(U)$, we have
\[\varphi_1(f)(a'_{\bar{g}}\otimes_{B'}m^*\otimes_B u)=a'_{\bar{g}}m^*u^{-1}_{\bar{g}}\otimes_Bf(u_{\bar{g}}\otimes_B u),\]
hence, for $f=\theta(c)\in E(U)$ we get
\[\varphi_1(\theta(c))(a'_{\bar{g}}\otimes_{B'}m^*\otimes_B u)=a'_{\bar{g}}m^*u^{-1}_{\bar{g}}\otimes_B u_{\bar{g}}c\otimes_B u.\]
On the other hand, $c':=\varphi_2(c)\in C_{A'}(B')_h$, hence, via the identification given by the isomorphism $\beta$, we have
\begin{align*} \theta'(\varphi_2(c))&(a'_{\bar{g}}\otimes_{B'}m^*\otimes_Bu) = a'_{\bar{g}}c'm^*u_{\bar{h}}^{-1}u_{\bar{g}}^{-1}\otimes_B u_{\bar{g}}u_{\bar{h}}\otimes_B u \\
   &= a'_{\bar{g}}\tilde\varphi(\sum_jm^*_jc\otimes_B m_j)   m^*u_{\bar{h}}^{-1}u_{\bar{g}}^{-1}\otimes_B u_{\bar{g}}u_{\bar{h}}\otimes_B u \\
   &= a'_{\bar{g}}  \sum_jm^*_jc \psi( m_j\otimes_{B'} m^*)   u_{\bar{h}}^{-1}u_{\bar{g}}^{-1}\otimes_B u_{\bar{g}}u_{\bar{h}}\otimes_B u \\
   &= a'_{\bar{g}}  \sum_j m^*_j \psi( m_j\otimes_{B'} m^*) u_{\bar{g}}^{-1}u_{\bar{g}} c u_{\bar{h}}^{-1}u_{\bar{g}}^{-1}\otimes_B u_{\bar{g}}u_{\bar{h}}\otimes_B u \\
   &= a'_{\bar{g}} \varphi( \sum_j m^*_j\otimes_B m_j) m^* u_{\bar{g}}^{-1}\otimes_B u_{\bar{g}} c u_{\bar{h}}^{-1}u_{\bar{g}}^{-1} u_{\bar{g}}u_{\bar{h}}\otimes_B u \\
   &= a'_{\bar{g}}  m^* u_{\bar{g}}^{-1}\otimes_B u_{\bar{g}} c \otimes_B u.
\end{align*}
Thus the statement is proved.
\end{proof}

\section{The butterfly theorem for \texorpdfstring{$\bar G$-}{group }graded Morita equivalences} \label{s:Butterfly}

\begin{subsec}\normalfont Let $N$ be a normal subgroup of $G$, $G'$ a subgroup of $G$, and $N'$ a normal subgroup of $G'$. We assume that $N'=G'\cap N$ and $G=G'N$, hence $\oG:=G/N\simeq G'/N'$.  Let $b\in Z(\O N)$ and $b'\in Z(\O N')$ be $\oG$-invariant block idempotents. We denote
\[ A:=b\O G, \qquad A':=b'\O G',\qquad B:=b\O N, \qquad B': = b'\O N'. \]
Then $A$ and $A'$ are strongly $\oG$-graded algebras, with  1-compo\-nents $B$ and $B'$ respectively.

Additionally, assume that $C_G(N)\subseteq G'$, and  denote $\bar{C}_G(N):=NC_G(N)/N.$ We consider the algebras
\[\xymatrix@C+=3cm{
	A:=b\O G & A':=b'\O G'\\
	C:=b\O NC_G(N) \ar@{-}[r]_{\sim} \ar@{-}[u] & C':=b'\O N'C_G(N) \ar@{-}[u]\\
	B:=b\O N \ar@{-}[r]^{\tensor*[_{B}^{}]{M}{_{B'}^{}}}_{\sim} \ar@{-}[u] & B':=b'\O N'. \ar@{-}[u]
}\]
\end{subsec}

If $M$ induces a Morita equivalence between $B$ and $B'$, the question that arises, is what can we deduce without the additional hypothesis that $M$ extends to a $\Delta$-module. One answer is given by the following proposition.

\begin{prop} \label{prop:extension_to_C} Assume that:
\begin{enumerate}
\item  $C_G(N)\subseteq G'$.
\item  $M$ induces a Morita equivalence between $B$ and $B'$.
\item  $zm=mz$ for all $m\in M$ and $z\in Z(N)$.
\end{enumerate}
Then there is a $\bar{C}_G(N)$-graded Morita equivalence between $C$ and $C'$, induced by the $\bar{C}_G(N)$-graded $(C,C')$-bimodule
\[C\tens{B}M\simeq M\tens{B'}C'\simeq (C\tens{} C'^{\mathrm{op}})\tens{\Delta(C\tens{}C'^{\mathrm{op}})}M.\]
\end{prop}

\begin{proof} Firstly, it is easy to see that our assumption implies that $NC_G(N)/N$ is isomorphic to $N'C_G(N)/N'$.  Thus both $C$ and $C'$ are indeed strongly $\bar{C}_G(N)$-graded algebras.
	
Now, we prove that there is a $\bar{C}_G(N)$-graded Morita equivalence between $C$ and $C'$.  It suffices to prove that $C\tens{B}M$ is actually a $\bar{C}_G(N)$-graded $(C,C')$-bimodule.
	
By Lemma $\ref{prop:theta}$, there exists a $\oG$-graded algebra homomorphism between $C_A(B)$ and $\text{End}_A(A\tens{B}M)^{\mathrm{op}}$. Moreover, note that $A\tens{B}M$ is a $\oG$-graded $(A,\text{End}_A(A\tens{B}M)^{\mathrm{op}})$-bimodule, hence by restricting the scalars we obtain that $A\tens{B}M$ is a $\oG$-graded $(A,C_A(B))$-bimodule. We truncate to the subgroup $\bar{C}_G(N)$ of $\oG$, and we obtain that $A_{\bar{C}_G(N)}\tens{B}M$ is a $\bar{C}_G(N)$-graded $(A_{\bar{C}_G(N)},C_A(B)_{\bar{C}_G(N)})$-bimodule, but $A_{\bar{C}_G(N)}=b\O NC_G(N)=C$, hence $\hat{M}:=C\tens{B}M$ is a $\bar{C}_G(N)$-graded $(C,C_A(B)_{\bar{C}_G(N)})$-bimodule.
	
We have  that $\O C_G(N)$ is $\bar{C}_G(N)$-graded with the 1-component $\O Z(N)$, and there is an algebra homomorphism from $\O C_G(N)$ to $C_A(B)$, whose image is evidently included in  $C_A(B)_{\bar{C}_G(N)}$. Hence, by restricting the scalars we obtain that $\hat{M}$ is a $\bar{C}_G(N)$-graded $(C,\O C_G(N))$-bimodule. Finally, since $M$ is $(B,B')$-bimodule, where $B'=b'\O N'$, we may define on $\hat{M}$ a structure of a $\bar{C}_G(N)$-graded $(C,b'\O N'C_G(N))$-bimodule, as follows. Let $c\in C$, $m\in M$, $c'\in C_G(N)\subseteq C'$ and $n\in N$, and define
\[(c\otimes m)c'n=cc'\otimes mn.\]
To see that this is well-defined, let $z\in Z(N)$, so $c'n=(c'z)(z^{-1}n)$. Then, by assumption (3), we have
\[(c\otimes m)(c'z)(z^{-1}n)=cc'z\otimes mz^{-1}n=cc'\otimes zmz^{-1}n=cc'\otimes mn.\]
Consequently, $\hat{M}$ is a $\bar{C}_G(N)$-graded $(C,C')$-bimodule.
\end{proof}

Our main result is a version for Morita equivalences of the so-called ``butterfly theorem" \cite[Theorem 2.16]{ch:Spath2018}.

\begin{theorem} \label{th:butterfly} Let $\hat G$ be another group with normal subgroup $N$, such that the block $b$ is also $\hat G$-invariant. Assume that:
\begin{enumerate}
\item $C_G(N)\subseteq G'$,
\item $\tilde M$ induces  a $\oG$-graded Morita equivalence between $A$ and $A'$;
\item $zm=mz$ for all $m\in M$ and $z\in Z(N)$.
\item the conjugation maps $\varepsilon:G\to \text{Aut}(N)$ and $\hat{\varepsilon}:\hat{G}\to \mathrm{Aut}(N)$ satisfy $\varepsilon(G)=\hat{\varepsilon}(\hat{G})$.
\end{enumerate}
Denote $\hat{G}'=\hat{\varepsilon}^{-1}(\varepsilon(G'))$. Then there is a $\hat G/N$-graded Morita equivalence between $\hat{A}:=b\O \hat{G}$ and $\hat{A}':=b'\O \hat{G}'$.

\end{theorem}
\begin{proof}
Consider the following diagram:
$$
\xymatrix@C+=1cm{
	\hat{A}:=b\O \hat{G} & A:=b\O G \ar@{-}[r]^{\tilde{M}}_{\sim} & A':=b'\O G' & \hat{A}':=b'\O \hat{G}'\\
	b\O NC_{\hat{G}}(N) \ar@{-}[u] & b\O NC_{G}(N) \ar@{-}[u] \ar@{-}[r]_{\sim}& b'\O N'C_{G}(N) \ar@{-}[u] & b'\O N'C_{\hat{G}}(N) \ar@{-}[u]\\
	 &B:=\O N b \ar@{-}[u] \ar@{-}[r]^{M}_{\sim} \ar@{-}[ul]& B':=\O N' b'. \ar@{-}[u] \ar@{-}[ur]&
}
$$

By the proof of \cite[Theorem 2.16]{ch:Spath2018}, we have that
$C_{\hat{G}}(N)\leq \hat{G}'$, $\hat{G}= N \hat{G}'$ and $N'=N\cap \hat{G}'$.
Note that $NC_G(N)$ is the kernel of the map $G\to\text{Out}(N)$ induced by conjugation. Hence, the hypothesis $\varepsilon(G)=\hat{\varepsilon}(\hat{G})$ implies that $G/NC_G(N)\simeq \hat{G}/NC_{\hat{G}}(N)$. It follows that $\oG/\bar{C}_G(N)\simeq \bar{\hat{G}}/\bar{C}_{\hat{G}}(N)$.

Let $C$ and $C'$ as in Proposition \ref{prop:extension_to_C} and denote $\hat{C}=b\O N C_{\hat{G}}(N)$ and $\hat{C}'=b'\O N'C_{\hat{G}'}(N)$.
By Proposition \ref{prop:extension_to_C}, we know that the Morita equivalence between $B$ and $B'$ induced by $M$ extends to a $\bar{C}_{\hat{G}}(N)$-graded Morita equivalence between $\hat{C}$ and $\hat{C}'$, induced by $\hat{C}\tens{B}M$.

Let $\T\subseteq G'$ be a complete set of representatives for the cosets of $N'C_G(N)$ in $G'$. Because $G=NG'$, $\T$ is a complete set of representatives for the cosets of $NC_G(N)$ in $G$.

For any $t\in\T$, we choose $\hat{t}\in \hat{G}'$ such that $\varepsilon(t)=\hat{\varepsilon}(\hat{t})$. Thus, we obtain a complete set $\hat{\T}$ of representatives of $N'C_{\hat{G}}(N)$ in $\hat{G}'$, so $\hat{\T}$ is also a complete set of representatives for the cosets of $NC_{\hat{G}}(N)$ in $\hat{G}$.

We need to define a $\hat\Delta:=\Delta(\hat{A}\otimes \hat{A}'^{\mathrm{op}})$-module structure on $M$, knowing that $M$ is $\Delta(A\otimes A'^{\mathrm{op}})$-module and a $\Delta(\hat{A}_{\bar{C}_{\hat{G}}(N)}\otimes \hat{A}_{\bar{C}_{\hat{G}}(N)}'^{\mathrm{op}})$-module, where
\[\Delta(\hat{A}_{\bar{C}_{\hat{G}}(N)}\otimes \hat{A}_{\bar{C}_{\hat{G}}(N)}'^{\mathrm{op}})\simeq \Delta(\hat{A}\otimes \hat{A}'^{\mathrm{op}})_{\bar{C}_{\hat{G}}(N)}.\]

We define $(\hat{t}\otimes \hat{t}^\circ)\cdot m=(t\otimes t^\circ)\cdot m$. It is a routine to verify that this definition does not depend on the choices we made, and that it gives the required $\hat\Delta$-module structure on $M$.

Alternatively, one may argue as follows: The cohomology class $[\hat\alpha]$ from $ H^2(\hat G/N, Z(B)^\times)$ associated to the $\hat\Delta_1$-module $M$ satisfies $\mathrm{Res}^{\hat G/N}_{\bar C_{\hat G}(N)}[\hat\alpha]=1$, because $M$ extends to a $\hat \Delta_{\bar{C}_{\hat{G}}(N)}$-module. It follows that $[\hat \alpha]\in\mathrm{Im}\mathrm{Inf}^{\hat G}_{NC_{\hat G}(N)}$. On the other hand, the class $[\alpha]\in H^2(\bar G, Z(B)^\times)$ associated to the $\Delta_1$-module $M$ is trivial, since $M$ extends to a $\Delta$-module. It is easy to see that
\[(t\otimes t^\circ)\otimes M \simeq (\hat t\otimes \hat t^\circ)\otimes M\] as $(B,B)$-bimodules, and since $G/NC_G(N)\simeq \hat G/NC_{\hat G}(N)$, we deduce that $[\hat\alpha]$ is also trivial, hence $M$ extends to a $\hat\Delta$-module.
\end{proof}


\vspace{10pt}

\hspace{-4mm}{\small{}}\\

\vspace{-25.3pt}
\ \hfill \
\begin{tabular}{c}
{\small\em Babe\cb{s}-Bolyai University} \\
{\small\em Department of Mathematics}\\
{\small\em Str. Mihail Kog\u alniceanu nr. 1}\\
{\small\em  400084 Cluj-Napoca, Romania}\\
{\small\em E-mail: {\tt marcus@math.ubbcluj.ro}} \\
{\small\em E-mail: {\tt minuta.aurelian@math.ubbcluj.ro}}
\end{tabular}

\end{document}